\documentclass{amsart}

\usepackage{amsmath,amssymb,amsthm}
\usepackage{multirow}

\newtheorem{theorem}{Theorem}
\newtheorem{lemma}{Lemma}

\newtheorem{conjecture}{Conjecture}

\allowdisplaybreaks

\begin{document}

\title[$k$-additive uniqueness of squares for multiplicative functions]{$k$-additive uniqueness of the set of squares\\ for multiplicative functions}
\author{Poo-Sung Park}
\address{Department of Mathematics Education, Kyungnam University, Changwon, Republic of Korea}
\email{pspark@kyungnam.ac.kr}

\thanks{This work was supported by ????}


\maketitle

\begin{abstract}
P. V. Chung showed that there are many multiplicative functions $f$ which satisfy $f(m^2+n^2) = f(m^2)+f(n^2)$ for all positive integers $m$ and $n$. In this article, we show that if more than $2$ squares in the additive condition are involved, then such $f$ is uniquely determined. That is, if a multiplicative function $f$ satisfies
\[
f(a_1^2 + a_2^2 + \dotsb + a_k^2) = f(a_1^2) + f(a_2^2) + \dotsb + f(a_k^2)
\]
for arbitrary positive integers $a_i$, then $f$ is the identity function. In this sense, we call the set of all posotive squares a \emph{$k$-additive uniqueness set} for multiplicative functions.
\end{abstract}

\section{Introduction}

In 1992, Claudia Spiro \cite{Spiro} called a set $E$ an \emph{additive uniqueness set} for a set $S$ of arithmetic functions if $f \in S$ is determined uniquely to be the identity function by the additive condition $f(m+n) = f(m)+f(n)$ for arbitrary $m,n \in E$. She proved that the set of primes is an additive uniqueness set for the set of multiplicative functions which do not vanish at some prime.

Many mathematicians have generalized her theorem. Fang \cite{Fang} showed that the additive condition for primes can be changed to three primes. Also, Dubickas and \v{S}arka \cite{D-S} extended the result to the arbitrary many primes. That is, the condition can be changed to 
\[
f(p_1+p_2+\dotsb+p_k) = f(p_1) + f(p_2) + \dotsb + f(p_k)
\]
for all primes $p_i$'s when $k \ge 2$. In their result, let us call the additive condition \emph{$k$-additivity} and the set of primes a \emph{$k$-additive uniqueness set}.

Some mathematicians have studied other additive uniqueness sets. Chung classified all multiplicative functions which are $2$-additive on positive squares \cite{Chung}. If $S$ is the set of completely multiplicative nonnegative functions, then the set of positive squares is a $2$-additive uniqueness set for $S$. But, it is not a $2$-additive uniqueness set for the set of multiplicative functions. Similarly, Ba\v{s}i\'c classified all arithmetic functions $f$ satisfying $f(m^2+n^2) = f(m)^2+f(n)^2$ \cite{Basic}. He showed that there exists a function $2$-additive on positive squares other than the identity function although the functions are restricted to multiplicative functions. The author generalized Bas\v{s}i\'c's result to
\[
f(a_1^2 + a_2^2 + \dotsb + a_k^2) = f(a_1)^2 + f(a_2)^2 + \dotsb + f(a_k)^2
\]
for all positive integers $a_i$'s and showed that such $f$ is determined uniquely to be the identity function when $k \ge 3$ \cite{Park}.

In this article, we show that the set of positive squares is $k$-additive uniqueness set for multiplicative functions when $k \ge 3$:

\begin{theorem}
Fix $k \ge 3$. The set of positive squares is the $k$-additive uniqueness set for multiplicative functions. That is, if a multiplicative function $f$ satisfies
\[
f(a_1^2 + a_2^2 + \dotsb + a_k^2) = f(a_1^2) + f(a_2^2) + \dotsb + f(a_k^2)
\]
for arbitrary positive integers $a_i$, then $f$ is the identity function.
\end{theorem}

This is the first example that a $3$-additive uniqueness set is not a $2$-additive uniqueness set.

The proof is similar to that of the author's generalization \cite{Park} of Ba\v{s}i\'c. But, since the $k$-additivity condition in Theorem 1 produces relations about $f(n^2)$ instead of $f(n)^2$, the proof is more difficult. It is known that the numbers inexpressible as sums of $k$ positive squares are infinite when $k = 3,4$ and are finite when $k \ge 5$.

We divide the proof into four sections: \S2 for $k=3$, \S3 for $k=4$, \S4 for $k=5$, and \S5 for $k \ge 6$.

\section{Proof for $3$-additivity}

Assume that a multiplicative function $f$ is $3$-additive on positive squares. That is, $f(a^2+b^2+c^2) = f(a^2)+f(b^2)+f(c^2)$ for arbitrary positive integers $a, b, c$.

\begin{lemma}
$f(n)=n$ for $1 \le n \le 12$ and $n=25$.
\end{lemma}
\begin{proof}
It is trivial that $f(3) = 3$ from $f(1)=1$. Solving the system of equations
\begin{align*}
f(6) 
&= f(2)\,f(3) = 3f(2) \\*
&= f(1^2+1^2+2^2) = f(1)+f(1)+f(4) = 2+f(4) \\
f(9)
&= f(1^2+2^2+2^2) = 1+2f(4) \\*
f(11)
&= f(1^2+1^2+3^2) = 2+f(9) \\*
f(14)
&= f(2)\,f(7) \\*
&= f(1^2+2^2+3^2) = 1+f(4)+f(9) \\
f(18)
&= f(2)\,f(9) \\*
&= f(1^2+1^2+4^2) = 2+f(16) \\
f(21)
&= f(3)\,f(7) = 3f(7) \\*
&= f(1^2+2^2+4^2) = 1+f(4)+f(16) \\
f(22)
&= f(2)\,f(11) \\*
&= f(2^2+3^2+3^2) = f(4) + 2f(9),
\end{align*}
we obtain that
\begin{align*}
&f(2)=2, &&f(3)=3, &&f(4)=4, &&f(7)=7, \\
&f(9)=9, &&f(11)=11, &&f(16)=16.
\end{align*}

From the equalities
\begin{align*}
f(27)
&= f(3^2+3^2+3^2) = 3f(9) = 27 \\*
&= f(1^2+1^2+5^2) = 2+f(25)
\end{align*}
we obtain that $f(25) = 25$. Then, the equalities
\begin{align*}
f(30)
&= f(2)\,f(3)\,f(5) = 6f(5) \\*
&= f(1^2+2^2+5^2) = 5+f(25) = 30
\end{align*}
yields $f(5)=5$. Also, $f(8)=8$ follows from the equalities
\begin{align*}
f(24)
&= f(3)\,f(8) = 3f(8) \\*
&= f(2^2+2^2+4^2) = f(4)+f(4)+f(16) = 24.
\end{align*}

Thus, we are done.
\end{proof}

\begin{lemma}\label{lem:2-3-5}
$f(2^n) = 2^n$, $f(3^n) = 3^n$, $f(5^n) = 5^n$ for all positive integer $n$.
\end{lemma}
\begin{proof}
It is easy to show $f(3^n)=3^n$ because $f(3)=3$, $f(3^2)=3^2$, and 
\begin{align*}
f(3^{2m+1}) &= f(3\cdot3^{2m}) = f\big((3^m)^2 + (3^m)^2 + (3^m)^2 \big)
= 3\,f(3^{2m})\\*
f(3^{2m+2}) &= f\big(9 (3^m)^2 \big) = f\big( (3^m)^2 + 2^2(3^m)^2 + 2^2(3^m)^2 \big) = 3^2\,f(3^{2m}).
\end{align*}

Now, consider $f(2^n)$. From the equalities
\begin{align*}
f(9\cdot2^{2m})
&= 9\,f(2^{2m}) \\*
&= f(2^{2m}+2^2(2^m)^2+2^2(2^m)^2) = f(2^{2m}) + 2f(2^{2m+2})
\end{align*}
we obtain that $f(2^{2m+2}) = 2^2\,f(2^{2m})$. Then, the following equality
\begin{align*}
f(6 \cdot 2^{2m})
&= f(3)\,f(2^{2m+1}) = 3\,f(2^{2m+1}) \\*
&= f\big( (2^m)^2 + (2^m)^2 + 2^2(2^m)^2 \big) = 2\,f(2^{2m}) + f(2^{2m+2})
\end{align*}
yields $f(2^{2m+1}) = 2\,f(2^{2m})$.

Similarly, from the equalities
\begin{align*}
f(27\cdot5^{2m})
&= f(27)\,f(5^{2m}) = 27\,f(5^{2m}) \\*
&= f\big( (5^m)^2 + (5^m)^2 + 5^2(5^m)^2 \big) = 2\,f(5^{2m}) + f(5^{2m+2}) \\
f(30\cdot5^{2m})
&= f(6)\,f(5^{2m+1}) = 6\,f(5^{2m+1}) \\*
&= f\big( (5^m)^2 + 2^2(5^m)^2 + 5^2(5^m)^2 \big) = 5\,f(5^{2m}) + f(5^{2m+2})
\end{align*}
we can conclude that $f(5^n) = 5^n$.
\end{proof}

\begin{lemma}[Hurwitz \cite{Hurwitz}, \protect{\cite[Chapter 6]{Grosswald}}]\label{lem:Hurwitz}
The only squares that are not sums of three positive squares are the integers $4^s$ and $25\cdot4^s$ with $s \ge 0$.
\end{lemma}

Now we prove the main theorem for $k=3$ by using induction. Suppose that $n > 12$ and $f(m)=m$ for all $m < n$. If $n$ is not a prime power, then $n = ab$ with $2 \le a,b < n$ and thus $f(n)=n$ by induction hypothesis. Thus, we may assume that $n$ is a prime power.

If $n$ is expressible as a sum of three positive squares, $f(n)=n$ by induction hypothesis. Otherwise, there are two cases:
\begin{enumerate}
\item if $n = a^2 + b^2 + c^2$, then at least one of $a,b,c$ should vanish, or
\item $n \ne a^2 + b^2 + c^2$ for any $a,b,c \in \mathbb{Z}$.
\end{enumerate}

Consider case (1). Assume $n = a^2+b^2$ for some positive $a$ and $b$. If $n$ is divisible by $5$, then $n$ should be a power of $5$. Thus, $f(n) = n$ by Lemma \ref{lem:2-3-5}.

Assume that $n$ is not divisible by $5$. If $a = 5^s t$ with $s \ge 0$ and $5 \nmid t$, then 
\[
f\big( (5a)^2 \big) = f( 5^{2s+2} t^2 ) = f(5^{2s+2})\,f(t^2) = 5^{2s+2} t^2 = (5a)^2
\]
since $t^2 < a^2 < n$ and $f(5^{2n+2}) = 5^{2n+2}$ by Lemma \ref{lem:2-3-5}. Similary, $f\big( (3b)^2 \big) = (3b)^2$ and $f\big( (4b)^2 \big) = (4b)^2$.

Note that $25$ and $n$ are relatively prime. So,
\begin{align*}
f(25n) 
&= f(25)\,f(n) = 25\,f(n) \\*
&= 
\left\{
\begin{array}{l}
f\big( (5a)^2 + (5b)^2 \big) = f\big( (5a)^2 + (3b)^2 + (4b)^2 \big) \\*
= f\big( (5a)^2 \big) + f\big( (3b)^2 \big) + f\big( (4b)^2 \big) \\*
= (5a)^2 + (3b)^2 + (4b)^2 = 25n
\end{array}
\right.
\end{align*}
and thus $f(n)=n$.

Now, if $n=a^2$, then we may assume that $n = 4^s$ with $s \ge 0$ by Lemma \ref{lem:Hurwitz}. It follows from Lemma \ref{lem:2-3-5} that $f(n)=n$.

Let us consider case (2). If $n$ is not expressible as a sum of three squares, then $n = 4^s(8t+7)$ with nonnegative integers $s$ and $t$. We may assume that $n = 8t+7$. Then, note that $2n = 2(8t+7) = 8(2t+1)+6$ is expressible as a sum of three squares. Thus, $f(2n) = 2n$. Since $n$ is odd, $f(2n)=f(2)\,f(n) = 2f(n)$.

From the above results, we can conclude that $f$ is the identity function.

\section{Proof for $4$-additivity}

In this section, assume that a multiplicative function $f$ is $4$-additive on positive squares. That is, $f(a^2+b^2+c^2+d^2) = f(a^2)+f(b^2)+f(c^2)+f(d^2)$ for arbitrary positive integers $a, b, c, d$.

Dubouis showed that which numbers are expressible as sums of positive squares as follows:
\begin{lemma}[Dubouis \cite{Dubouis}, \protect{\cite[Chapter 6]{Grosswald}}]\label{lem:Dubouis}
Every integer $n$ can be represented as a sum of $k$ positive squares except
\[
n = \begin{cases}
1, 3, 5, 9, 11, 17, 29, 41, 2\cdot4^m, 6\cdot4^m, 14\cdot4^m ~(m \ge 0) & \text{if } k = 4, \\
33 & \text{if } k = 5, \\
1, 2, \dots, k-1, k+1, k+2, k+4, k+5, k+7, k+10, k+13 & \text{if } k \ge 5.
\end{cases}
\]
\end{lemma}

\begin{lemma}
$f(n)=n$ for $n = 1,3,5,9,11,17,29,41$.
\end{lemma}

\begin{proof}
We have that $f(1) = 1$ and $f(4)=4$. Since
\begin{align*}
f(10)
&= f(2)\,f(5) \\*
&= f(1^2+1^2+2^2+2^2) = 2f(1)+2f(4) = 10 \\
f(12)
&= f(4)\,f(3) = 4f(3) \\*
&= f(1^2+1^2+1^2+3^2) = 3+f(9) \\
f(15)
&= f(3)\,f(5) \\*
&= f(1^2+1^2+2^2+3^2) = 6+f(9) \\
f(18)
&= f(2)\,f(9) \\*
&= f(1^2+2^2+2^2+3^2) = 9+f(9) \\
f(20)
&= f(4)\,f(5) = 4f(5) \\*
&= f(1^2+1^2+3^2+3^2) = 2+2f(9),
\end{align*}
we can deduce that
\begin{align*}
& f(2) = 2,
&& f(3)= 3,
&& f(5) = 5,
&& f(9) = 9.
\end{align*}

Also, we have that $f(16)=f(2^2+2^2+2^2+2^2)=16$. Then, $f(11) = 11$ from 
\[
f(22) = f(2)\,f(11) = f(1^2+1^2+2^2+4^2) = 22.
\] 
Similarly, we can deduce that $f(17)=17$, $f(29)=29$, and $f(41)=41$.
\end{proof}

We can compute $f(n)=n$ inductively for other $n$ except $2\cdot4^m$, $6\cdot4^m$, and $14\cdot4^m$. Thus, since $f(5\cdot2\cdot4^m)=5\cdot2\cdot4^m$ can be also computed inductively,
\begin{align*}
f(2\cdot4^m) &= 2\cdot4^m \\
f(6\cdot4^m) &= f(3)\,f(2\cdot4^m) = 6\cdot4^m \\
f(14\cdot4^m) &= f(7)\,f(2\cdot4^m) = 14\cdot4^m.
\end{align*}
 
We are done.

\section{Proof for $5$-additivity}

In this section, assume that a multiplicative function $f$ is $5$-additive on positive squares. That is, $f(a_1^2+a_2^2+a_3^2+a_4^2+a_5^2) = f(a_1^2)+f(a_2^2)+f(a_3^2)+f(a_4^2)+f(a_5^2)$ for arbitrary positive integers $a_i$'s.

By Lemma \ref{lem:Dubouis} every positive integer can be represented as sums of five positive squares except for $1, 2, 3, 4, 6, 7, 9, 10, 12, 15, 18, 33$.

\begin{lemma}
$f(n) = n$ for $1 \le n \le 16$.
\end{lemma}

\begin{proof}
We have that $f(1)=1$ and $f(5)=5$. Solving the system of equations
\begin{align*}
f(8) &= f(1^2+1^2+1^2+1^2+2^2) = 4+f(4) \\
f(11) &= f(1^2+1^2+1^2+2^2+2^2) = 3+2f(4) \\
f(13) &= f(1^2+1^2+1^2+1^2+3^2) = 4+f(9) \\
f(14) &= f(2)\,f(7) \\*
&= f(1^2+1^2+2^2+2^2+2^2) = 2+3f(4) \\
f(16) &= f(1^2+1^2+1^2+2^2+3^2) = 3+f(4)+f(9) \\
f(20) &= f(4)\,f(5) = 5f(4) \\*
&= f(1^2+1^2+1^2+1^2+4^2) = 4+f(16) \\
f(21) &= f(3)\,f(7) \\*
&= f(1^2+1^2+1^2+3^2+3^2) = 3+2f(9) \\
f(22) &= f(2)\,f(11) \\*
&= f(1^2+2^2+2^2+2^2+3^2) = 1+3f(4)+f(9) \\
f(24) &= f(8)\,f(3) \\*
&= f(1^2+1^2+2^2+3^2+3^2) = 2+f(4)+2f(9) \\
f(26) &= f(2)\,f(13) \\*
&= f(1^2+1^2+2^2+2^2+4^2) = 2+2f(4)+f(16),
\end{align*}
we obtain that $f(n)=n$ for $1 \le n \le 16$.
\end{proof}

It is clear that $f(18)=f(2)\,f(9)=18$ and $f(33) = f(3)\,f(11) = 33$. Since other numbers can be expressible as sums of $5$ positive squares, we can conclude that $f$ is the identity function by induction.

\section{Proof for $k$-additivity with $k \ge 6$}

Let $f_k$ be $k$-additive on positive squares. That is, $f_k(a_1^2+a_2^2+\dotsb+a_k^2) = f_k(a_1^2) + f_k(a_2^2) + \dotsb + f_k(a_k^2)$ for arbitrary positive integers $a_i$'s.

By Lemma \ref{lem:Dubouis} every integer can be expressible as sums of $k$ positive squares except for $1, 2, \dotsc, k-1, k+1, k+2, k+4, k+5, k+7, k+10, k+13$. Note that $f_k(1)=1$ and $f_k(k)=k$.

\begin{lemma}\label{lem:f_k(n)_for_n=4,9,16}
$f_k(n^2)=n^2$ for $n=2,3,4$.
\end{lemma}

\begin{proof}
Using equalities
\begin{align*}
28 
&= 1^2 + 3^2 + 3^2 + 3^2 \\*
&= 2^2 + 2^2 + 2^2 + 4^2 \\
20 
&= 1^2 + 1^2 + 1^2 + 1^2 + 4^2 \\*
&= 2^2 + 2^2 + 2^2 + 2^2 + 2^2 \\
40
&= 1^2 + 1^2 + 1^2 + 1^2 + (2\cdot3)^2 \\*
&= 2^2 + 3^2 + 3^2 + 3^2 + 3^2,
\end{align*}
we obtain that two cases
\[
f_k(4)=f_k(9)=f_k(16)=1 \qquad\text{or}\qquad
f_k(4) = 4,
f_k(9) = 9,
f_k(16) = 16.
\]

Note that
\[
f_6(9) = f_6(1^2+1^2+1^2+1^2+1^2+2^2) = 5+f_6(4) 
\]
yields that $f_6(4)=4$, $f_6(9)=9$, and $f_6(16)=16$. But, we cannot find $k$-additivity relations between $f_k(4)$, $f_k(9)$, and $f_k(16)$ for general $k$. So we need another way to exclude $f_k(4)=f_k(9)=f_k(16)=1$.

Suppose $f_k(4)=f_k(9)=f_k(16)=1$. Let us compute $f(5^2)$. Note that $5^2+1^2+1^2+1^2+1^2+1^2-4^2 = 14$ can be expressible as a sum of $5$ positive squares $\le 4^2$ by Lemma \ref{lem:Dubouis}. Since $f_k(n^2)=n^2$ for $n \le 4$, the equality
\[
f_k(5^2+\underbrace{1^2+\dotsb+1^2}_{5\text{ summands}} + \underbrace{1^2+\dotsb+1^2}_{k-6\text{ summands}}) 
= f_k(4^2 + a_1^2 + \dotsb + a_5^2 + \underbrace{1^2+\dotsb+1^2}_{k-6\text{ summands}}) 
\]
with $1 \le a_i \le 4$ yields
\[
f_k(5^2) + (k-1) = f(4^2) + f(a_1^2) + \dotsb + f(a_5^2) + (k-6) = k
\]
and we can, thus, find $f_k(5^2)=1$.

Hence we obtain $f_k(n^2)=1$ for all $n$ inductively from the similar equality
\begin{align*}
n^2 + \underbrace{1^2+\dotsb+1^2}_{5\text{ summands}} = (n-1)^2 + a_1^2 + \dotsb + a_5^2
\end{align*}
with $1 \le a_i \le n-1$. But, we must be careful for $n=7$ since $7^2+5-6^2 = 18$ cannot be expressible as a sum of $5$ positive squares by Lemma \ref{lem:Dubouis}. So, in this case, we use another equality $7^2 + 1^2 + 1^2 + 1^2 + 1^2 + 2^2 - 6^2 = 21$. 

But, the determination $f_k(k^2) = f_k\big( (k-2)^2 \big) = f_k(4) = 1$ makes a contradiction:
\begin{align*}
1 = f_k(k^2) &= f_k\big( (k-2)^2 + (k-1)2^2 \big) \\*
&= f_k\big((k-2)^2\big) + (k-1)\,f_k(4) = 1 + (k-1) = k.
\end{align*}
This excludes the case $f_k(4)=f_k(9)=f_k(16)=1$. That is, we can determine that $f_k(2^2) = 2^2$, $f_k(3^2) = 3^2$, $f_k(4^2) = 4^2$.
\end{proof}

\begin{lemma}
$f_k(n^2)=n^2$ for all $n$.
\end{lemma}

\begin{proof}
The method used in the proof of the Lemma \ref{lem:f_k(n)_for_n=4,9,16} shows that
\begin{align*}
&f_k(n^2) + (k-1)\,f_k(1^2) \\
&= f_k\big( (n-1)^2 \big) + f_k(a_1^2) + \dotsb + f_k(a_5^2) + (k-6)\,f_k(1^2)
\end{align*}
with $1 \le a_i \le n-1$. Thus, $f_k(n^2) = n^2$ is obtained inductively for all $n$.
\end{proof}

Now, let $n$ be an integer $\ge 2$. If $n$ is expressible as a sum of $k$ positive squares, we are done. So, assume that $n$ cannot be expressed as a sum of $k$ positive squares. If an integer $m$ is sufficiently large, $m$ is expressible as a sum of $k$ positive squares by Lemma \ref{lem:Dubouis}. So is $nm$. Thus, $f_k(m)=m$ and $f_k(nm) = nm$. If we choose $m$ to be relatively prime to $n$,
\begin{align*}
nm = f_k(nm) = f_k(n)\,f_k(m) = f_k(n)\,m
\end{align*}
and thus $f_k(n)=n$.

Therefore, proof completes.

%
%
%
%
%
%
%
%
%
%
%
%

\section*{Acknowledgment}

The author would like to thank East Carolina University for its support and hospitality.

\end{document}